\documentclass[12pt]{amsart}
\usepackage{latexsym, amsbsy, amsmath, amsfonts, amssymb, amsthm, amscd, amscd}
\usepackage[all]{xy}
\usepackage[PostScript=dvips]{diagrams}

\usepackage{verbatim}
\newcommand{\Z}{\mathbb Z}

\newtheorem{Theorem}{Theorem}[section]

\newtheorem{Corollary}{Corollary}[section]
\newtheorem{Proposition}{Proposition}[section]

\allowdisplaybreaks[2]
\numberwithin{equation}{section}
\numberwithin{figure}{section}
\title{On Vassiliev invariants of braid groups of the sphere 
}
\author[Kaabi]{N.~Kaabi}
\address{L'Institut Supérieur des Etudes Technologiques (ISET) du Kef,
Campus Universitaire, Boulifa, 7100 Le Kef, Tunis} 
\author[Vershinin]{V.~V.~Vershinin}

\address{D\'epartement des Sciences Math\'ematiques,
                                    Universit\'e Montpellier II,
Place Eug\`ene Bataillon,
34095 Montpellier cedex 5, France}
\email{ vershini@math.univ-montp2.fr}

\address{Sobolev Institute of Mathematics, Novosibirsk 630090,
Russia } 
\email{ versh@math.nsc.ru}


\begin{document}
\begin{abstract}
We construct a universal Vassiliev invariant for braid groups of the sphere and
the mapping class groups of the sphere with $n$ punctures. The case of a 
sphere is different from  the classical braid groups or braids of oriented
surfaces of genus strictly greater than zero, since  Vassiliev
invariants in a group without  2-torsion do not distinguish elements
of braid group  of a sphere.
\end{abstract}
\maketitle
\tableofcontents

\section{Introduction}
The theory of Vassiliev  (or finite type) invariants starts with the works of 
V.~A.~Vasiliev \cite{Va1, Va2} 
though the ideas which lie in the foundations of this theory 
can be found in the work of M.~Gousarov \cite{Gu}. 
The basic idea is classical in Mathematics: to introduce a filtration in 
a complicated fundamental object such that the corresponding associated 
graded object is simpler and sometimes possible to describe. The construction of
spectral sequences have similar features. A lot of progress had been done
during the last couple of decades in the theory of Vassiliev invariants 
of knots, the basic object of the study. Also similar  constructions
were done for braids. In 1996 T.~Stanford \cite{St} established a relation
between Vassiliev invariants and the lower central series of the pure braid group.
Later S.~Papadima \cite{Pa} constructed
a universal Vassiliev invariant (``Kontsevich integral") over $\Z$ for classical braids. A similar   construction was done by J.~Mostovoy and S.~Willerton \cite{MW}. 
For braids on oriented closed surfaces  of genus $g\geq 1$ this was done by
J.~Gonz\'ales-Meneses and L.~Paris \cite{GMP}. This universal Vassiliev invariant
is not multiplicative. It was shown by Bellingeri and Funar \cite{BelF}
a multiplicative universal  invariant  does not exist in the case of positive genus ($g\geq 1$).  
In the present work we consider the case of the sphere $S^2$. This case  is interesting because the braid
groups on the sphere contains torsion and Vassiliev invariants in rational numbers
do not give a complete set of invariants. Namely M.~Eisermann \cite{Ei1, Ei2} gave the example of a two different  elements 
in $B_n(S^2)$ which are
not distinguished by Vassiliev invariants in rational numbers.  Our universal invariant
is not multiplicative. Usually multiplicative universal Vassiliev invariant
is constructed over rational numbers, which is not good for the sphere, but on the other hand there is no interdiction of existence of such an invariant over the 
 integers.

We also study the mapping class group of the sphere with $n$ punctures, this group is
closely connected with the braid group (on $n$-strands) of the sphere. From algebraic point of view
the mapping class group of the sphere with $n$ punctures is obtained from 
the braid group of the sphere if we add one more relation (\ref{eq:sphe_mc}) 
to the presentation of the braid group of the sphere.
It turns out that for these two types of groups  as for the  braid groups examined
before the study of
Vassiliev invariants, topological by their nature  reduces to purely algebraic 
constructions and facts.

The authors are thankful to Paolo Bellingeri for valuable remarks. 

This work was finished during the visit of the second author to the IHES, so he would like to express his great appreciation to the administration of the Institute for
the invitation and the excellent conditions of work.

\section{Braid groups of the sphere and the mapping class groups of
the  sphere with $n$ punctures 
}

Let $M$ be a topological space and let $M^n$ be the $n$-fold Cartesian product
of $M$. The \textit{$n$-th ordered configuration space} $F(M,n)$ is defined
by
$$
F(M,n)=\{(x_1,\ldots, x_n)\in M^n \ | \ x_i\not=x_j \textrm{ for }
i\not=j\}
$$
with subspace topology of $M^n$. The symmetric group $\Sigma_n$ acts
on $F(M,n)$ by permuting coordinates. The orbit space $$B(M,n)=F(M,n)/\Sigma_n$$ is called the \textit{$n$-th unordered configuration space} or simply  \textit{$n$-th  configuration space}.
The \textit{braid group} $B_n(M)$ is defined to be the fundamental group
$\pi_1(B(M,n))$. The \textit{pure braid group} $P_n(M)$ is defined to be the fundamental
group  $\pi_1(F(M,n)$. From the covering $F(M,n)\to F(M,n)/\Sigma_n$,
we get a short exact sequence of groups
\begin{equation}
\{1\}\to P_n(M)\to B_n(M)\to \Sigma_n\to \{1\}.
\label{eq:br_ex_sec} 
\end{equation}
We will use later the following classical Fadell-Neuwirth Theorem.
\begin{Theorem}\cite{FN}\label{theorem3.1}
For $n>m$ the coordinate projection (forgetting of $n-m$ coordinates)
$$
\delta^{(n)}_m\colon F(M,n)\to F(M,m), \ (x_1,\ldots,x_n)\mapsto (x_1,\ldots,x_m)
$$
is a fiber bundle with fiber $F(M\smallsetminus Q_{m}, n-m)$, where $Q_{m}$ is a set of $m$ distinct points in $M$.
\end{Theorem}
In this work we consider the case $M=S^2$ and classical braids which are braids of the disc: $M=D^2$.
 
Usually  the braid group of the disc $Br_n= B_n(D^2)$ 
is given by the following Artin presentation \cite{Art1}.
It has the generators $\sigma_i$, 
$i=1, ..., n-1$, and two types of relations: 
\begin{equation}
 \begin{cases} \sigma_i \sigma_j &=\sigma_j \, \sigma_i, \ \
\text{if} \ \ |i-j| >1,
\\ \sigma_i \sigma_{i+1} \sigma_i &= \sigma_{i+1} \sigma_i \sigma_{i+1} \, .
\end{cases} \label{eq:brelations}
\end{equation}

The generators $a_{i,j}$, $1\leq i<j\leq n $ of the pure
braid group $P_n$ (of a disc) can be described as elements of the braid 
group $Br_n$  by the formula:
$$a_{i,j}=\sigma_{j-1}...\sigma_{i+1}\sigma_{i}^2\sigma_{i+1}^{-1}...
\sigma_{j-1}^{-1}.$$ 
The defining relations among $a_{i,j}$, which are called the \emph{Burau
relations} (\cite{Bu1}, \cite{Ma}) are as follows: 
\begin{equation}
\begin{cases}
a_{i,j}a_{k,l}=a_{k,l}a_{i,j}
\ \text {for} \ i<j<k<l \ \text {and} \ i<k<l<j, \\
a_{i,j}a_{i,k}a_{j,k}=a_{i,k}a_{j,k}a_{i,j} \ \text {for} \
i<j<k, \\
a_{i,k}a_{j,k}a_{i,j}=a_{j,k}a_{i,j}a_{i,k} \ \text
{for} \ i<j<k, \\
a_{i,k}a_{j,k}a_{j,l}a_{j,k}^{-1}=a_{j,k}a_{j,l}a_{j,k}^{-1}a_{i,k}
\ \text {for} \ i<j<k<l.\\
\end{cases}
\label{eq:burau}
\end{equation}

\smallskip

It was proved by O.~Zariski \cite{Za1} and then
rediscovered by E.~Fadell and J.~Van Buskirk \cite{FaV} that a presentation of the braid group
on a 2-sphere can be given with the   generators 
$\sigma_i$, $i=1, ..., n-1$, the same as for the classical braid group, satisfying
the braid relations (\ref{eq:brelations})
and the following sphere relation: 
\begin{equation}
\sigma_1 \sigma_2 \dots \sigma_{n-2}\sigma_{n-1}^2\sigma_{n-2} \dots
\sigma_2\sigma_1 =1.
\label{eq:spherelation}
\end{equation}
 
Let $\Delta$ be the Garside's   fundamental element   in the braid 
group $Br_{n}$ \cite{Gar}. It can be expressed in particular by the following
word in canonical generators:
$$\Delta = \sigma_1 \dots \sigma_{n-1} \sigma_1 \dots \sigma_{n-2} \dots  
\sigma_1 \sigma_2 \sigma_1.$$
If we use Garside's notation $\Pi_t\equiv \sigma_1\dots \sigma_t$, 
$1\leq t\leq n-1$, then
$\Delta \equiv \Pi_{n-1} \dots \Pi_1$.

For the pure braid group on a 2-sphere let us introduce the elements
$a_{i,j}$ for all $i, j$ by the formulas:  
\begin{equation}
\begin{cases}
a_{j,i}= a_{i,j} \ \ \text{for} \ i<j\leq n,\\
a_{i,i}= 1. 
\label{eq:aji}
\end{cases}
\end{equation}
The pure braid group on a 2-sphere has the generators $a_{i,j}$
which satisfy the Burau relations (\ref{eq:burau}), the relations (\ref{eq:aji}),
 and the following relations \cite{GVB}:
\begin{equation*}
a_{i,i+1}a_{i,i+2} \dots a_{i,i+n-1} = 1 \ \ \text{for all} \ i\leq n,\\
\end{equation*}
with the convention that  indices are considered $\mod n$: $k+n =k$.
Note that $\Delta^2$ is a pure braid and can be expressed by the following
formula
\begin{multline*}
\Delta^2 = (a_{1,2}a_{1,3} \dots a_{1,n})(a_{2,3}a_{2,4} \dots a_{2,n})
\dots (a_{n-1,n}) = \\
(a_{1,2})(a_{1,3} a_{2,3})( a_{1,4}a_{2,4}a_{3,4}) \dots (a_{1,n} \dots
\dots a_{n-1,n}).
\end{multline*}
This element of the braid group generates its center.

Another object of our study is the mapping class groups of the sphere with $n$
punctures.
The (general) mapping class group is an  important object in Topology,
Complex Analysis, Algebraic Geometry and other domains. 
It is a rare situation when the method of Algebraic Topology works
perfectly well: the application of the functor of fundamental group
completely solves the topological problem: group of isotopy classes of
homeomorphisms
is described in terms of automorphisms of the fundamental group
of the corresponding surface, as  the Dehn-Nilsen-Baer theorem states 
(see \cite{Iv}, for example).

Let  $S_{g,b,n}$ be an oriented surface of genus $g$ with $b$ boundary 
components and we remind that  $Q_n$ denotes a set of $n$ punctures (marked points)
in the surface. Consider the group 
$\operatorname{Homeo}(S_{g,b,n})$ 
of orientation preserving
self-homeomorphisms of $S_{g,b,n}$ which fix pointwise the boundary (if
$b>0$) and map the set $Q_n$ into itself. 
 
Let 
$\operatorname{Homeo}^0(S_{g,b,n})$ be the normal subgroup of 
self-homeomorphisms of $S_{g,b,n}$ which are isotopic to identity. 
Then the {\it mapping class group} $ {M}_{g,b,n}$ is defined as a 
quotient group

\begin{equation*}
 {M}_{g,b,n} = \operatorname{Homeo}(S_{g,b,n})/
\operatorname{Homeo}^0(S_{g,b,n}).
\end{equation*}
These groups are closely connected with braid groups.
W.~Magnus in \cite{Mag1} interpreted the $n$-braid group as the mapping class group of an $n$-punctured disc with the  fixed boundary:
\begin{equation*}
Br_n \cong {M}_{0,1,n}. 
\end{equation*}
Like braid groups  the groups
$ {M}_{g,b,n}$ has a natural epimorphism to the symmetric group  
$\Sigma_{n}$ with the kernel called the  {\it pure mapping class group} ${PM}_{g,b,n}$,
so there exists an exact sequence:
\begin{equation*}
1 \to {PM}_{g,b,n} \to  {M}_{g,b,n} \to \Sigma_{n}\to 1.
\end{equation*}
Geometrically the pure mapping class group $ {PM}_{g,b,n}$ 
consists of isotopy classes of homeomorphisms that preserve the  punctures pointwise.

In the present work we consider the pure mapping class group 
$ {PM}_{0,0,n}$ of a punctured 2-sphere (so the genus is equal to 0) with
no boundary components that we simply denote by $ {PM}_{n}$; the same way
we denote further $ {M}_{0,0,n}$ simply by $ {M}_{n}$. 

The group $ {PM}_{n}$ is closely related to the pure braid group
$P_n(S^2)$ on the 2-sphere as well as its non-pure analogue 
$ {M}_{n}$
is related with the (total) braid group $B_n(S^2)$ on the 2-sphere. 

 W.~Magnus obtained in \cite{Mag1} (see also 
\cite{MKS}) a presentation of the 
mapping class group $ {M}_{n}$ for the $n$-punctured 2-sphere. 
It has the same generators as $B_n(S^2)$ and a complete set of relations consists
of (\ref{eq:brelations}), (\ref{eq:spherelation})
 and  the following relation 
\begin{equation}
(\sigma_1 \sigma_2 \dots \sigma_{n-2}\sigma_{n-1})^n =1.
\label{eq:sphe_mc}
\end{equation}
This defines an epimorphism
\begin{equation*}
\gamma:B_n(S^2) \to M_n.
\end{equation*} 
For the generators $\sigma_1,$ $\sigma_2$, $\dots$, $\sigma_{n-2},$
$\sigma_{n-1}$,
 subject to the braid relations (\ref{eq:brelations}) the 
condition (\ref{eq:sphe_mc}) is equivalent to the following relation
\begin{equation*}
\Delta^2 =1.
\label{eq:sphe_mcD}
\end{equation*}

Using Theorem~\ref{theorem3.1} we have the following morphism of fibrations
\begin{equation}
\begin{diagram}
F(D^2\smallsetminus\{p_1,p_2\}, n-2)&\rTo^{i} &F(D^2,n)&\rTo^
{\delta_2}& F(D^2,2)\\
\dTo & & \dTo
&&\dTo
\\
F(S^2\smallsetminus\{p_1,p_2,p_3\}, n-2)&\rTo^{i}&F(S^2,n+1) &\rTo^{\delta_3}& F(S^2,3),\\
\end{diagram}
\label{eq:3.1}
\end{equation}
where the vertical lines are induced by an inclusion of a disc into the sphere.
Let us denote by $P_n(S^2_3)$ the pure braid group on $n$ strands of a 2-sphere 
with three points deleted or equivalently the subgroup of the pure braid group 
on $n+2$ strands of a disc where (say, the last) two strands are fixed
as trivial (unbraided) strands which is also equal to the fundamental group 
of $F(S^2\smallsetminus\{p_1,p_2,p_3\}, n)$. 

The following statement follows from the normal forms of the groups
$P_n(S^2)$ and ${PM}_{n}$ given in \cite{GVB} and on the geometrical level it 
follows from diagram (\ref{eq:3.1}) and it was expressed in \cite{GG2}. Note that the groups $P_2(S^2)$ and ${PM}_{3}$ are trivial.
\begin{Theorem} (i) The pure braid group on a 2-sphere $P_n(S^2)$ for $n\geq 3$
is isomorphic to the direct product
of the cyclic group $C_2$ of order 2 (generated by $\Delta^2$) and ${PM}_{n}$.
\par
(ii) The pure braid group $P_n$ for $n\geq 2$ is isomorphic to the direct product
of the infinite cyclic group $C$ (generated by $\Delta^2$) and ${PM}_{n+1}$. 
\par 
(iii) The groups ${PM}_{n}$ and $P_{n-3}(S^2_3)$ are isomorphic for 
$n\geq 4$.
\par
(iv) There is a commutative diagram of groups and homomorphisms 
\begin{equation}
\begin{diagram}
P_n\cong& PM_{n+1}\times C\\
\dTo>{\rho_p}&\dTo>{\delta\times\rho}\\
P_n(S^2)\cong& PM_n\times C_2,\\
 \end{diagram} \label{eq:pmc}
\end{equation} 
where $\rho_p$ is the canonical epimorphism $P_n\to P_n(S^2)$,
$\delta$ is induced by the  Fadell-Neuwirth  fibration via the isomorphism of item (iii) and $\rho$ is the canonical epimorphism
of the infinite cyclic group onto the cyclic group of order 2. 
\label{theo:gil-vb} 
\end{Theorem} 

The isomorphism of the part (i) of Theorem~\ref{theo:gil-vb} is compatible 
with the homomorphisms
$p_i:P_n(S^2)\to P_{n-1}(S^2)$ and
$pm_i : {PM}_{n} \to {PM}_{n-1}$  
consisting of deleting of one strand or forgetting one point, that means that the following diagram is commutative 
\begin{equation*}
\begin{diagram}
P_n(S^2)\cong& PM_{n}\times C_2\\
\dTo>{p_i}&\dTo>{{pm_i}\times{id}}\\
P_{n-1}(S^2)\cong& PM_{n-1}\times C_2.\\
 \end{diagram} 
\end{equation*}

\section{Lie algebras from descending central series of groups 
}

Lie algebras obtained from the filtration of descending central series of 
the pure braid groups are essential ingredients in the construction
of the universal Vassiliev invariant. We describe such Lie algebras
for the groups $P_n(S^2)$ and $\mathcal {PM}_{0,n}$ in this section.
We will use them in the next section in our construction of universal
invariant.

For a group $G$ the descending central series 
\begin{equation*}
G =\Gamma_1  > \Gamma_2 > \dots  > \Gamma_i > \Gamma_{i+1} > \dots .
\end{equation*}
\noindent
is defined by the formula
\begin{equation*}
\Gamma_1 = G, \ \ \Gamma_{i+1} =[\Gamma_{i}, G].
\end{equation*}
This series of   groups gives rise to the 
associated graded Lie algebra (over $\Z$) $gr^*_\Gamma(G)$ 
\begin{equation*}
gr^i_\Gamma(G)= \Gamma_i/\Gamma_{i+1}.
\end{equation*}

A presentation of the Lie algebra $gr^*_\Gamma(P_n)$ for the pure braid group
was done in the work of T.~Kohno \cite{K}, and
can be described as follows. It is the quotient of the
free Lie algebra $L[A_{i,j}| \, 1 \leq i < j \leq n]$ generated by
elements $A_{i,j}$ with $1 \leq i < j \leq n$ modulo the
``infinitesimal braid relations" or ``horizontal $4T$ relations"
given as follows:

\begin{equation}
\begin{cases}
 [A_{i,j}, A_{s,t}] = 0, \  \text{if} \ \{i,j\} \cap \{s,t\} = \phi, \\
 [A_{i,j}, A_{i,k} + A_{j,k}] = 0, \  \text{if} \ i<j<k , \\
  [A_{i,k}, A_{i,j} + A_{j,k}] = 0, \ \text{if} \ i<j<k. \\
  \end{cases}
  \label{eq:kohno}
\end{equation}
It is convenient sometimes to have conventions like (\ref{eq:aji}).
So let us introduce the generators $A_{i,j}, \, 1 \leq i ,j \leq n$,
not necessary $i < j$, by the formulae
\begin{equation*}
\begin{cases}
A_{j,i} =  A_{i,j} \ \ \text{for} \ \ 1\leq i < j \leq n ,\\
A_{i,i} = 0 \ \ \text{for all} \ \ 1\leq i \leq n .\\
\end{cases}
\end{equation*}
For this set of generators the defining relations (\ref{eq:kohno})
can be rewritten as follows
\begin{equation}
\begin{cases}
A_{i,j} =  A_{j,i} \ \ \text{for} \ \ 1\leq i,j \leq n ,\\
A_{i,i} = 0 \ \ \text{for} \ \ 1\leq i \leq n ,\\
 [A_{i,j}, A_{s,t}] = 0, \  \text{if} \ \{i,j\} \cap \{s,t\} = \phi, \\
 [A_{i,j}, A_{i,k} + A_{j,k}] = 0 \ \  \text{for all different} \ \ i,j,k .\\
\end{cases}
    \label{eq:kohn_sym}
\end{equation}

Y.~Ihara in \cite{Ih} gave  a presentation of the Lie algebra 
$gr^*_\Gamma(P_n(S^2))$ 
of the pure braid group of a sphere. 
It is the quotient of the
free Lie algebra $L[B_{i,j}| \, 1 \leq i ,j \leq n]$ generated by
elements $B_{i,j}$ with $1 \leq i , j \leq n$ modulo the
 following  relations:
\begin{equation}
\begin{cases}
B_{i,j} =  B_{j,i} \ \ \text{for} \ 1\leq i,j \leq n , \\
B_{i,i} =  0 \ \ \text{for} \ 1\leq i \leq n , \\
 [B_{i,j}, B_{s,t}] = 0, \  \text{if} \ \{i,j\} \cap \{s,t\} = \phi, \\
 \sum_{j=1}^n B_{i,j} = 0, \  \text{for} \ 1\leq i \leq n. \\
    \end{cases}
    \label{eq:ihara}
\end{equation}

\smallskip

\noindent
It is also a quotient algebra of the Lie algebra $gr^*_\Gamma(P_n)$: the  relations 
of the last type in 
(\ref{eq:kohn_sym}) are the consequences of the third and the forth type 
relations in (\ref{eq:ihara}).

We shall use the following results which were proved in \cite{KV}.
\begin{Theorem} (i) The graded Lie algebra $gr^*_\Gamma( {PM}_{n})$ 
is the quotient of the
free Lie algebra $L[B_{i,j}| \, 1 \leq i ,j \leq n]$ modulo the
 following  relations:
\begin{equation}
\begin{cases}
B_{i,j} =  B_{j,i} \  \text{for} \ 1\leq i,j \leq n , \\
B_{i,i} =  0 \ \ \text{for} \ 1\leq i \leq n , \\
 [B_{i,j}, B_{s,t}] = 0, \  \text{if} \ \{i,j\} \cap \{s,t\} = \phi, \\
 \sum_{j=1}^n B_{i,j} = 0, \  \text{for} \ 1\leq i \leq n, \\
 \medskip
 \sum_{i=1}^{n-1}\sum_{j=i+1}^n B_{i,j} = 0. \\ 
    \end{cases}
    \label{eq:m0n}
\end{equation} 
(ii) The graded Lie algebra $gr^*_\Gamma( {PM}_{n})$ is the
quotient of the free Lie algebra $L[B_{i,j}| \, 1 \leq i , j \leq
n-1]$ generated by the elements 
$B_{i,j}, \, 1 \leq i , j \leq n-1,$ (smaller number of generators than in (i)) modulo the following relations:
\begin{equation}
\begin{cases}
B_{i,j} - B_{j,i}=0 \ \  \text{for} \ \ 1\leq i, j \leq n-1 , \\
B_{i,i} = 0 \ \ \text{for} \ \ 1\leq i \leq n-1 , \\
[B_{i,j}, B_{s,t}] = 0,  \ \ \text{if} \ \ \{i,j\} \cap \{s,t\} = \emptyset , \\
[B_{i,j}, B_{i,k} + B_{j,k}] = 0  \ \ \text{for all different} \ \ i,j,k , \\
 \smallskip
\sum_{i=1}^{n-2}\sum_{j=i+1}^{n-1} B_{i,j} = 0 . \\
\end{cases}
  \label{eq:m0n2}
\end{equation}
  \label{theo:lie_m0n}
\end{Theorem}
\begin{Corollary}
A presentation of the Lie algebra $gr^*_\Gamma(P_n(S^2))$ can be given with generators
$A_{i,j}$ with $1 \leq i< j  \leq n-1$,  modulo 
the following  relations:
\begin{equation*}
\begin{cases}
 [A_{i,j}, A_{s,t}] = 0, \  \text{if} \ \ \{i,j\} \cap \{s,t\} = \phi, \\
 \smallskip
 [A_{i,j}, A_{i,k} + A_{j,k}] = 0  \ \ \text{for all different} \ \ i,j,k , \\
\smallskip 
2  (\sum_{i=1}^{n-2}\sum_{j=i+1}^{n-1} A_{i,j}) = 0. 
 \end{cases}
\end{equation*}
So the element $\sum_{i=1}^{n-2}\sum_{j=i+1}^{n-1} A_{i,j}$ of order 2 generates
the central sub-algebra in  $gr^*_\Gamma(P_n(S^2))$. 
\end{Corollary}

\section{Universal Vassiliev invariants for ${M}_{n}$  and $B_n(S^2)$
}
This section is central in our work: the universal Vassiliev invariants
for the groups $P_n(S^2)$ and $\mathcal {PM}_{0,n}$ are constructed here.

We sketch briefly the basic ideas of the theory of Vassiliev invariants
for braids. For the classical braids (i.e. of a disc) it can be found, for example  in \cite{Pa}.

Let $A$ be an abelian group, then the group $V$ of all maps (non necessary 
homomorphisms) from $B_n(S^2)$ to $A$ is called the group of {\it invariants} 
of $B_n(S^2)$:
\begin{equation*} 
V=\operatorname{Map}(B_n(S^2), A).
\end{equation*}
If $A$ is a commutative ring then $V$ becomes an $A$-module. 

Let $\Z[B_n(S^2)]$ be the group ring of the group  $B_n(S^2)$, then 
\begin{equation*} 
\operatorname{Map}(B_n(S^2), A)= \operatorname{Hom}( \Z[B_n(S^2)], A).
\end{equation*}
where $\operatorname{Hom}( \Z[B_n(S^2)], A)$ is an abelian group of 
homomorphisms of the group $\Z[B_n(S^2)]$ into the  group $A$.
 
We can enlarge an invariant
$v\in V$ for singular braids using the rule 
\begin{equation*} 
v({\text{singular crossing of $i$-th and $i+1$ strands}}) =v(\sigma_i)-v(\sigma_i^{-1}).
\end{equation*}
The elements $\sigma_i-\sigma_i^{-1}\in \Z[B_n(S^2)]$, $i=1, \dots, n-1$,
 generate an ideal of the ring  $\Z[B_n(S^2)]$ which we denote by $W$;
 degrees of this ideal define a multiplicative
filtration ({\it Vassiliev filtration}) $W^m= \Phi^m( \Z[B_n(S^2)])$. An invariants $v\in V$ is called
{\it of degree} $m$ if $v(x)=0$ for all $x\in \Phi^{m+1}(\Z[B_n(S^2)])$. So the group
$V_m$ of invariants of degree $m$ is defined as
\begin{equation*} 
V_m= \operatorname{Hom}( \Z[B_n(S^2)]/\Phi^{m+1}(\Z[B_n(S^2)]) , \, A).
\end{equation*}
The advantage of braids is that this filtration can be characterized completely algebraically. Let $S$ be a map 
 from the symmetric group $\Sigma_n$: 
\begin{equation*} 
S:\Sigma_n\to B_n(S^2)
\label{eq:sect}\end{equation*}
which is a section of the canonical epimorphism
$
 B_n(S^2)\to \Sigma_n
$
(\ref{eq:br_ex_sec}). 
It is not  a homomorphism which does not exist with such a condition.
We can set up, for example, $S(s_i)=\sigma_i$.
A similar map 
\begin{equation*} 
S_M:\Sigma_n\to M_n
\label{eq:sectM}\end{equation*}
 is a section of the canonical epimorphism
$
 M_n\to \Sigma_n.
$

Let $I$ be the augmentation ideal of the  group ring $\Z[P_n(S^2)]$. The powers of $I$
generate a filtration of the ring $\Z[P_n(S^2)]$ and hence of the ring 
$\Z[P_n(S^2)]\otimes \Z[\Sigma_n]$ as we assume that elements of $\Z[\Sigma_n]$ have zero filtration.
The same filtration we have in $\Z[ {M}_{n}]$.
\begin{Proposition} There is an isomorphism of abelian groups with filtration
\begin{equation*}
\Z[ {P}_{n}(S^2)]\otimes \Z[\Sigma_n]\cong \Z[ {B}_{n}(S^2)], \\
\end{equation*}
\begin{equation*}
\Z[ {PM}_{n}]\otimes \Z[\Sigma_n]\cong \Z[ {M}_{n}], \\
\end{equation*}
which are induced by the canonical inclusions of the pure groups and the maps $S$
and $S_M$;
the rings $\Z[ {B}_{n}(S^2)]$ and  $\Z[ {M}_{n}]$ are equipped with Vassiliev filtration.
\label{Proposition:tensor_s}
\end{Proposition}
\hfill $\square$


Let $c$ be the generator of the infinite cyclic group $C$  and let $\Z[C]$ be the group ring of $C$. We  denote by $C_2$ the cyclic group of the order 2 with the generator $a$, $\Z[C_2]$ is the
 group ring of $C_2$ and we define the homomorphism
\begin{equation*}
\rho: \Z[C]\to \Z[C_2],
\end{equation*}
by the formula
\begin{equation*}
\rho(c)=a.
\end{equation*}
\begin{Proposition} There are isomorphisms of rings
\begin{equation*}
\Z[ {P}_{n}]\cong \Z[ {PM}_{n+1}]\otimes \Z[C], 
\end{equation*}
\begin{equation*}
\Z[P_n(S^2)]\cong\Z[PM_n]\otimes \Z[C_2],
\end{equation*}
which can be included into the following commutative diagram
\begin{equation*}
\begin{diagram}
\Z[P_n]\cong&\Z[PM_{n+1}]\otimes \Z[C]\\
\dTo>{\rho_p}&\dTo>{{\delta}{\otimes}\rho}\\
\Z[P_n(S^2)]\cong&\Z[PM_n]\otimes \Z[C_2],\\
 \end{diagram}
\end{equation*} 
where the morphisms $\rho_p$ and $\delta$ in the diagram are induced by the corresponding morphisms of
the diagram~(\ref{eq:pmc}). 
\label{Proposition:ring_pr} 
\end{Proposition}
\begin{proof}
It follows from Theorem~\ref{theo:gil-vb}.
\end{proof}
\begin{Proposition} The intersections of Vassiliev filtration for $\Z[ {B}_{n}(S^2)]$ 
and $\Z[ {M}_{n}]$ are trivial
\begin{equation*}
\bigcap_{m\geq 0} \Phi^m(\Z[{B}_{n}(S^2)]) = 0, \quad \quad \bigcap_{m\geq 0} \Phi^m(\Z[M_n])=0.
\end{equation*}
The 
groups  $\Phi^{m}(\Z[M_n])/\Phi^{m+1}(\Z[M_n])$ are torsion free.
The group  ${PM}_{n}$ is residually torsion free nilpotent, the group $P_n(S^2)$ is
residually  nilpotent.
\label{Proposition:res_nil}
\end{Proposition}
\begin{proof} Proposition~\ref{Proposition:tensor_s} reduces the question
to the $I$-adic filtration of $\Z[ {PM}_{n}]$ and $\Z[ {P}_{n}(S^2)]$.
In the case of  $\Z[ {PM}_{n}]$ the situation is the same as for the 
pure braid group of a disc,  ${PM}_{n}$ is a subgroup of $P_{n-1}$
(cf Theorem~\ref{theo:gil-vb}) and hence it 
is residually torsion free nilpotent.
For the group ring $\Z[ {P}_{n}(S^2)]$ because of its  structure as the tensor  product
(Proposition~\ref{Proposition:ring_pr}) we have an isomorphism 
\begin{equation*}
I^m(\Z[P_n(S^2)])\cong\operatorname{Im}(\oplus_{k=0}^m I^{m-k}(\Z[PM_n])\otimes I^{k}(\Z[C_2])\to \Z[P_n(S^2)]),
\end{equation*}
where $\operatorname{Im}$ denotes as usual the image of the map given in parentheses.
This group ring  can be also presented as an
abelian group as the following
direct sum
\begin{equation}
\Z[P_n(S^2)]\cong\Z[PM_n]\oplus \Z[PM_n]\{(a-1)\}
\label{eq:PS_sum}
\end{equation}
where by $a$ we denoted the generator of cyclic group $C_2$ of order 2.
For an element $z\in \Z[PM_n]$ let
\begin{equation*}
z=\sum_i n_i b_i.
\end{equation*}
be a decomposition of $z$ with respect to the base of the free abelian group  $\Z[PM_n]$.
Let  the greatest common divisor of the set of numbers $\{n_i\}$
be equal to $2^k n$ where $n$ is coprime with 2.
Let $v(z)$ be the order function 
(see \cite[chapter~III, \S~2.2]{BouCA}, for example) of the $I$-adic filtration on $\Z[PM_n]$
then the order function $v_2(z)$ of the $I$-adic filtration on the second
summand in the right hand part of formula 
(\ref{eq:PS_sum}) is given as follows
\begin{equation*}
v_2(z(a-1))=v(z)+k+1.
\end{equation*}
The intersection $\bigcap_{k=0}^\infty  I^{k}(\Z[C_2])$ is trivial, the group 
$\Z[ {PM}_{n}]$ is residually torsion free nilpotent, so
\begin{equation*}
\bigcap_{k=0}^\infty  I^{k}(\Z[P_n(S^2)])=0.
\end{equation*}
We remind that integral dimension subgroups (for the group $P_n(S^2)$)
are defined by the formula
$$D_m(P_n(S^2))=P_n(S^2)\cap (1+I^m),$$
where the intersection is taken in $\Z[P_n(S^2)]$.
Hence the  intersection of 
$D_m(P_n(S^2)$
is trivial  the group $P_n(S^2)$
is residually  nilpotent (of course not residually torsion free nilpotent);
the last fact can be also seen directly.
\end{proof}

The filtered algebra $\mathcal{P} _n$ is defined as the universal enveloping 
algebra of the Lie algebra $gr^*_\Gamma(P_n)$ for the standard pure braid group
\begin{equation*}
\mathcal{P}_n= U(gr^*_\Gamma(P_n)).
\end{equation*}
Its completion
$\widehat{\mathcal{P}_n}$ is the target of the universal Vassiliev invariant for the pure 
braids \cite{Pa}
\begin{equation*}
\mu:\Z[ {P}_{n}]\to  \widehat{ {\mathcal{P}}_{n}}. 
\end{equation*}
Let $\mathcal{PM}_{n}$ be the universal enveloping algebra of the Lie algebra  
$gr^*_\Gamma( {PM}_{n})$; so as an associative algebra it has the generators which are
in one-to-one correspondence with the generators $B_{i,j}$ of 
$gr^*_\Gamma( {PM}_{n})$,
say it will be $x_{i,j}$, $ 1 \leq i ,j \leq n$, which satisfy the associative form of relations (\ref{eq:m0n2}). 
Also we denote by $\mathcal{P}_n(S^2)$ the universal enveloping algebra of the Lie algebra $gr^*_\Gamma( {P}_{n}(S^2))$.

As usual one can define a Hausdorff filtration (intersection is zero) on $\mathcal{PM}_{n}$ and on $\mathcal{P}_{n}(S^2)$ by giving a degree 1 to each generator $x_{i,j}$. 
The canonical epimorphism of groups $\rho_p:P_n\to P_n(S^2)$ induces an epimorphism
of filtered algebras
\begin{equation*}
\rho_a:\mathcal{P} _n\to \mathcal{P}_{n}(S^2).
\end{equation*}
 We denote by $\widehat {\mathcal{PM}}_{n}$  the completion 
 of $\mathcal{PM}_{n}$  with respect to the
topology, defined by this filtration. The same way
$\widehat{\mathcal{P}_{n}}(S^2)$ is the completion 
 of $\mathcal{P}_{n}(S^2)$.
The algebra $\widehat {\mathcal{PM}}_{n}$ can be also described as an algebra of non-commutative power series of $x_{i,j}$
factorized by the closed ideal generated by the left hand sides of relations
(\ref{eq:m0n2}). 

Let  $\widehat{\mathcal{A}}$ be an associative algebra with unit 
such that as an 
abelian group it is isomorphic to the direct sum of integers and 2-adic numbers $\Z\oplus\Z_2$. We denote the generator of the first summand by 1 and the generator of the second summand by $x$.  The multiplication in  $\widehat{\mathcal{A}}$
is given by the rule
\begin{equation*}
x^2=-2x.
\end{equation*}
This algebra is filtered as follows $\Phi^0=\widehat{\mathcal{A}}$, $\Phi^1 =\Z_2$, $\Phi^m$
is generated by $2^mx$, for $m=2,3 \dots$. 

We define the homomorphisms 
\begin{equation*}
\alpha: \Z[C_2)]\to \widehat{\mathcal{A}},
\end{equation*}
\begin{equation*}
\chi: \Z[C]\to\Z[[y]], 
\end{equation*}
\begin{equation*}
\beta: \Z[[y]]\to\widehat{ \mathcal{A}}
\end{equation*}
by the formulae
\begin{equation}
 \alpha(a)= 1+x, \quad 
\chi(c)= 1+y 
, \quad
\beta(y)=x.
\label{eq:acy}
\end{equation}

\begin{Proposition} The homomorphisms of rings $\alpha$ and $\chi$
respect the filtration and induce a multiplicative isomorphism 
at the associated graded level.
They fit in the following commutative diagram of homomorphisms of rings.
\begin{equation*}
\begin{diagram}
\Z[C]&\rTo{\rho} &  \Z[C_2] \\
\dTo>{\chi}&&\dTo>{\alpha}\\
\Z[[y]]&\rTo{\beta} & \widehat{\mathcal{A}}.\\
\end{diagram}\label{eq:diag2} 
\end{equation*}
\end{Proposition}
\begin{proof} 
It suffices to verify that
$$
(a-1)^2= a^2-2a+1= 2-2a=-2(a-1).
$$
\end{proof}
\begin{Proposition} There are isomorphisms of filtered rings
\begin{equation*}
\widehat{ {\mathcal{P} }}_{n}\cong \widehat{\mathcal {PM}}_{n+1}\widehat{\otimes}\Z[[y]], 
\end{equation*}
\begin{equation*}
\widehat{\mathcal {P}}_{n}(S^2)\cong\widehat{\mathcal {PM}}_{n}\widehat{\otimes} \widehat{\mathcal{A}},
\end{equation*}
which can be included into the following commutative diagram of filtered ring homomorphisms
\begin{equation*}
\begin{diagram}
\widehat{\mathcal{P} }_n\cong&\widehat{\mathcal{PM}}_{n+1}\widehat{\otimes}\Z[[y]]\\
\dTo>{\widehat{\rho_a}}&\dTo>{\widehat{\delta}\widehat{\otimes}\beta}\\
\widehat{\mathcal {P}}_{n}(S^2)\cong&\widehat{\mathcal {PM}}_{n}\widehat{\otimes} \widehat{\mathcal{A}},\\
 \end{diagram} \label{eq:diag0}
\end{equation*}
where the morphisms $\widehat{\rho_a}$ and $\widehat{\delta}$ in the  diagram are induced by the corresponding morphisms of
the diagram~(\ref{eq:pmc}). 
\end{Proposition}
\begin{proof}
The statements  follow from the facts 
about the direct product of the Lie algebras  similar to Theorem~\ref{theo:gil-vb}.
\end{proof}

The map 
\begin{equation*}
\kappa: \Z[ {PM}_{n}]\to \widehat{\mathcal {PM}}_{n}
\end{equation*}
can be defined following the same steps as the definition of the universal
Vassiliev invariant in \cite{Pa}.
However it is more simple to use the universal invariant from \cite[Theorem~1.1]{Pa}
and define $\kappa$ 
as  the following composition
\begin{equation*}
\Z[PM_{n+1}]\to\Z[ {P}_{n}]\stackrel{\mu}\rightarrow \widehat{ {\mathcal{P}}_{n}}
\to \widehat{ {\mathcal{PM}}}_{n+1},
\end{equation*}
where the first map is the canonical inclusion and the last one is the canonical projection.
We can also reason inversely: at first construct $\kappa$, then define the map
$\widehat{\kappa\otimes\chi}$ as the composition
\begin{equation*}
\Z[PM_{n+1}]\otimes \Z[C] \stackrel{\kappa\otimes\chi}\longrightarrow
\widehat{ {\mathcal{PM}}_{n+1}}{\otimes}\Z[[y]]\to \widehat{ {\mathcal{PM}}_{n+1}}\widehat{\otimes}\Z[[y]],
\end{equation*}
where the last map is the completion, 
and then define
$\mu$ using the following diagram  
\begin{equation}
\begin{diagram}
\Z[P_n]\cong&\Z[PM_{n+1}]\otimes \Z[C]\\
\dTo>{\mu}&\dTo>{\widehat{\kappa\otimes\chi}}\\
\widehat{\mathcal{P} }_n\cong&\widehat{\mathcal{PM}}_{n+1}\widehat{\otimes}\Z[[y]].\\
 \end{diagram} \label{eq:diag1}
\end{equation}
The map $\mu$ defined by (\ref{eq:diag1}) is a universal Vassiliev invariant for the 
classical braids, though it may not coincide with the map constructed in \cite{Pa}
which is not  unique.  
\begin{Theorem} The map 
\begin{equation*}
\kappa: \Z[ {PM}_{n}]\to \widehat{\mathcal {PM}}_{n}
\end{equation*}
respects the filtration and induces a multiplicative isomorphism 
at the associated graded level.
\end{Theorem} 
\begin{proof}
Two proofs  can be done. The first one is to follow the steps 
of the proof of the  same fact for the classical pure braid groups \cite{Pa}; this is
possible because the group  ${PM}_{n}$ has the similar structure as $P_n$: it is an iterated semi-direct product of free groups \cite{GVB}. In particular the algebra 
${\mathcal {PM}}_{n}$ has no torsion and the Quillen map \cite{Q}
\begin{equation}
{\mathcal {PM}}_{n} \to gr^*_I \Z[PM_n]
\end{equation}
to the graded object associated to the $I$-adic filtration of the ring $\Z[PM_n]$
becomes an isomorphism with $gr^*(\kappa)$ as its inverse.
Another proof is to apply the fact that $P_n$
is the product of ${PM}_{n}$ with the infinite cyclic group generated by the
center of $P_n$, and then use Proposition \ref{Proposition:ring_pr} and diagram 
(\ref{eq:diag1}).  
\end{proof}

We define the map 
\begin{equation*}
\lambda: \Z[P_n(S^2)]\to \widehat{\mathcal{P}_n}(S^2)
\end{equation*}
using the following diagram 
\begin{equation}
\begin{diagram}
 \Z[PM_n]\otimes \Z[C_2]&\cong\Z[P_n(S^2)] \\
\dTo>{\widehat{\kappa\otimes\alpha}}&\dTo>{\lambda} \\ \widehat{\mathcal{PM}}_n\widehat{\otimes}\widehat{\mathcal{A}}&\cong\widehat{\mathcal{P}_n}(S^2).\\
\end{diagram}\label{eq:kala} 
\end{equation}

\begin{Theorem} The map 
\begin{equation*}
\lambda: \Z[P_n(S^2)]\to \widehat{\mathcal{P}_n}(S^2)
\end{equation*}
respects the filtration, induces a multiplicative isomorphism 
at the associated graded level and fits in the following diagram of filtered rings
\begin{equation}
\begin{diagram}
\Z[P_n]&\rTo{\rho_p} & \Z[P_n(S^2)] \\
\dTo>{\mu}&&\dTo>{\lambda} \\
\widehat{\mathcal{P} }_n &\rTo{\widehat{\rho_a}} & \widehat{\mathcal{P}_n}(S^2).\\
\end{diagram}\label{eq:diag2.5} 
\end{equation}
\label{Theorem:un_vas_pu} 
\end{Theorem} 
\begin{proof}
The maps $\kappa$ and $\alpha$ respect the filtration and induce a multiplicative isomorphism 
at the associated graded level, so this is true for ${\widehat{\kappa\otimes\alpha}}$.
We continue diagram (\ref{eq:diag1}) 
\begin{equation*}
\begin{diagram}
\Z[P_n]&\cong\Z[PM_{n+1}]\otimes \Z[C]&\rTo{\delta\otimes\rho} & \Z[PM_n]\otimes \Z[C_2]&\cong\Z[P_n(S^2)] \\
\dTo>{\mu}&\dTo>{\widehat{\kappa\otimes\chi}}&&\dTo>{\widehat{\kappa\otimes\alpha}}&\dTo>{\lambda} \\
\widehat{\mathcal{P} }_n &\cong\widehat{\mathcal{PM}}_{n+1}\widehat{\otimes}\Z[[y]]&\rTo{\widehat{\delta}\widehat{\otimes}\beta} & \widehat{\mathcal{PM}}_n\widehat{\otimes}\widehat{\mathcal{A}}&\cong\widehat{\mathcal{P}_n}(S^2).\\
\end{diagram}\label{eq:diag3} 
\end{equation*}
Its outer frame is diagram (\ref{eq:diag2.5}).
\end{proof}

The symmetric group $\Sigma_n$ acts on the algebras $\widehat{\mathcal {PM}}_{n}$
and $\widehat{\mathcal{P}_n}(S^2)$ by the action on the indices of $x_{i,j}$:
\begin{equation*}
\sigma(x_{i,j})= x_{\sigma(i),\sigma(j)}, \quad \sigma\in\Sigma_n.
\end{equation*}
This action preserves the defining relations (\ref{eq:m0n2}) and (\ref{eq:ihara}).
We define  the following filtered algebras as the  semi-direct products with respect to
the given action:
\begin{equation}
\widehat{\mathcal{M}_n} =
\widehat{\mathcal {PM}}_{n}\rtimes\Z[\Sigma_n],
\label{eq:MS}
\end{equation}
\begin{equation}
\widehat{\mathcal{B}_n}(S^2) =
\widehat{\mathcal {P}_{n}}(S^2)\rtimes\Z[\Sigma_n].
\label{eq:BS}
\end{equation}

According to the Markov normal form for $B_n(S^2)$ proved by R.~Gillet and J.~Van Buskirk in \cite{GVB} every element $b$ of $B(S^2)$ 
can be written uniquely in the form
\begin{equation*}
b=q S(p),
\end{equation*}
where $q\in P_n(S^2)$ and $p$ is the permutation defined by the braid $b$. 
We define the map 
\begin{equation*}
K: \Z[B_n(S^2)]\to \widehat{\mathcal{B}_n}(S^2)
\end{equation*}
by the formula
\begin{equation}
K(b)= \lambda(q) \otimes p.
\label{eq:K}
\end{equation}
The map
\begin{equation*}
K_M: \Z[ {M}_{n}]\to \widehat{\mathcal {M}}_{n}
\end{equation*}
is defined similarly using $\kappa$ instead of $\lambda$ in (\ref{eq:K}).
\begin{Theorem} The homomorphisms of abelian groups
\begin{equation*}
K_M: \Z[ {M}_{n}]\to \widehat{\mathcal {M}}_{n},
\end{equation*}
\begin{equation*}
K: \Z[B_n(S^2)]\to \widehat{\mathcal{B}_n}(S^2)
\end{equation*}
are injections, they respect the filtration, induce a multiplicative isomorphisms 
at the associated graded level and fit in the following diagram of filtered rings
\begin{equation*}
\begin{diagram}
\Z[B_n(S^2)]&\rTo{\rho_p} & \Z[M_n] \\
\dTo>{K}&&\dTo>{K_M} \\
\widehat{\mathcal{B}}_n(S^2) &\rTo{\widehat{\rho_a}} & \widehat{\mathcal{M}_n},\\
\end{diagram}\label{eq:diag4.5} 
\end{equation*}
which leads to the following diagram with isomorphisms
\begin{equation}
\begin{diagram}
\Z[B_n(S^2)]\cong& \ \Z[M_{n}]\otimes \Z[C_2]\\
\dTo>{K}&\dTo>{\widehat{K_M{\otimes\alpha}}}\\
\widehat{\mathcal{B} }_n(S^2) \ \  \cong&\widehat{\mathcal{M}}_{n}
\widehat{\otimes}\widehat{\mathcal{A}}.\\
 \end{diagram} \label{eq:isoBS}
\end{equation}
\end{Theorem} 
\begin{proof}
It follows from Proposition~\ref{Proposition:tensor_s}, 
Theorem~\ref{Theorem:un_vas_pu} and definitions (\ref{eq:MS} -- \ref{eq:K}). 
Note that maps $\kappa$ and $\lambda$ that are involved in the definition of 
$K$ and $K_M$  are related by diagram (\ref{eq:kala}). The  injectivity follows from 
Proposition~\ref{Proposition:res_nil}. 
\end{proof}
\begin{Corollary}
The groups  $\Z[M_n]/\Phi^{m}(\Z[M_n])$ and $\widehat{\mathcal{M}_n}/\Phi^{m}(\widehat{\mathcal{M}_n})$ are isomorphic and are
torsion free. There are also isomorphisms of abelian groups.
\begin{multline}
\Z[B_n(S^2)]/\Phi^{m+1}(\Z[B_n(S^2)])\cong \widehat{\mathcal{B}_n}(S^2)/\Phi^{m+1}(\widehat{\mathcal{B}_n}(S^2))\cong \\
(\widehat{\mathcal{M}_n}/\Phi^{m+1}(\widehat{\mathcal{M}_n}))
\oplus(\bigoplus_{i=0}^{m-1}\Phi^{i}(\widehat{\mathcal{M}_n})/\Phi^{i+1}(\widehat{\mathcal{M}_n})
\otimes \Z/{2^{m-i}}).
\label{eq:iso_2i}
\end{multline}
\end{Corollary}
\begin{proof}
The groups  $\Z[M_n]/\Phi^{m}(\Z[M_n])$ and $\widehat{\mathcal{M}_n}/\Phi^{m}(\widehat{\mathcal{M}_n})$ are subgroups of corresponding groups for 
the pure braid groups of a disc, so they are finitely generated free 
abelian groups. Isomorphisms follow from the diagrams like (\ref{eq:diag1}). 
Last isomorphism of formula (\ref{eq:iso_2i}) follows from the bottom
isomorphism of the formula (\ref{eq:isoBS}). Namely, 
$\widehat{\mathcal { A}} = \Z\oplus\Z_2$, so 
$$
\widehat{\mathcal{M}}_{n}
\widehat{\otimes}\,\widehat{\mathcal{A}}\cong \widehat{\mathcal{M}}_{n}\oplus
\widehat{\mathcal{M}}_{n}
\widehat{\otimes}\,\Z_2 
$$
and hence
$$
(\widehat{\mathcal{M}}_{n}
\widehat{\otimes}\,\widehat{\mathcal{A}})/
\Phi^{m+1}(\widehat{\mathcal{M}}_{n}
\widehat{\otimes}\,\widehat{\mathcal{A}})\cong (\widehat{\mathcal{M}}_{n}/
\Phi^{m+1}(\widehat{\mathcal{M}_n}))\oplus
((\widehat{\mathcal{M}}_{n}
\widehat{\otimes}\,\Z_2 )/\Phi^{m+1}(\widehat{\mathcal{M}}_{n}
\widehat{\otimes}\,\Z_2 )).
$$
The quotient group $\widehat{\mathcal{M}_n}/\Phi^{m}(\widehat{\mathcal{M}_n})$ 
is a free abelian finitely generated group, so we choose  its finite basis 
$\{g_{i,k}\}$, such that $g_{i,k}\in \Phi^i(\widehat{\mathcal{M}}_{n})$. 
Then every element of  $(\widehat{\mathcal{M}}_{n}
\widehat{\otimes}\,\Z_2 )/\Phi^{m+1}(\widehat{\mathcal{M}}_{n}
\widehat{\otimes}\,\Z_2 )$
can be written in the form
$$
\sum_{i+j\le m}\alpha_{i,k,j} \, g_{i,k}\otimes 2^jx,
$$
where $\alpha_{i,k,j}$ can be equal to $0$ or $1$. As we have
$$
2^{m-i}(g_{i,k}\otimes x)= g_{i,k}\otimes 2^{m-i} x\in \Phi^{m+1}(\widehat{\mathcal{M}}_{n}
\widehat{\otimes}\,\Z_2 )
$$
the last isomorphism in (\ref{eq:iso_2i}) follows.
\end{proof}

\begin{Corollary}
There exist two elements in $B_n(S^2)$ which are not distinguished by any 
Vassiliev invariant to an abelian group without 2-torsion. 
For any couple of elements $a\not=b$  of $B_n(S^2)$ there exists a natural number $k$ and Vassiliev invariant $v$ to an abelian  group $A$ with 
 an element 
of order $2^k$ such that $v$ distinguish $a$ and $b$:
$$
v(a)\not= v(b).
$$   
\label{Corollary:not_dis} 
\end{Corollary}
M.~Eisermann in \cite{Ei1, Ei2} gave an example of a couple of elements 
in $B_n(S^2)$ which are
not distinguished by Vassiliev invariants to an abelian group without 2-torsion.  These
elements are: the trivial braid that represents the unit in the braid group
$B_n(S^2)$  and the braid $\tau= (\sigma_1 \sigma_2 \dots \sigma_{n-2}\sigma_{n-1})^n$ (cf. (\ref{eq:sphe_mc})).
In the context of isomorphisms of Propositions~\ref{Proposition:tensor_s}
and \ref{Proposition:ring_pr} these elements correspond 
to element $1$ and $a$ in $\Z[PM_n]$. By the formulae (\ref{eq:acy})
$a$ maps to $1+x\in \widehat{\mathcal{A}}$. The difference between $1$ and $1+x$
is equal to $x$ and according to (\ref{eq:iso_2i}) $x$ corresponds to a 
generator of
$\Z/2^k$ in 
$\Z[B_n(S^2)]/\Phi^{m+1}(\Z[B_n(S^2)])$2. So it is mapped to $0$ by any map to an abelian group without 
$2$-torsion. 

\section{Examples\label{sec:ex}}

1. $n=4.$ The pure braid group  $P_4(S^2)$ of a 2-sphere is isomorphic to the direct product
of the cyclic group of order 2 (generated by $\Delta^2$) and the pure braid group on one strand of a 2-sphere with three points deleted, it is the fundamental group 
of disc with two points deleted, that is a free group $F_2$ on two generators.
Its associated graded Lie algebra is a direct sum of central $\Z/2$ and the
free Lie algebra on two generators. The pure mapping class group 
${PM}_{4}$ is isomorphic to a free group on two generators. According 
to Theorem~\ref{theo:lie_m0n} its associated graded Lie algebra is the
 free Lie algebra
 on two generators. The universal Vassiliev invariant
 for ${PM}_{4}$ is nothing but Magnus expansion 
$\Z[F_2]\stackrel{\mu_e}\rightarrow \Z\langle\langle x_1, x_2\rangle\rangle $
 and the universal invariant
for  $P_4(S^2)$ is 
\begin{equation*}
 \Z[F_2]\otimes\Z[C_2] \stackrel{\widehat{\mu_e\otimes\alpha}}\rightarrow \Z\langle\langle x_1, x_2\rangle\rangle\widehat{\otimes}
 \widehat{\mathcal{A}}.
\end{equation*}

2. $n=5.$
 The pure mapping class group 
${PM}_{5}$ is isomorphic to a semi-direct product of a
free group on three generators and a free group on two generators.
$${PM}_{5}\cong F\langle a_{1,2},  a_{1,3},  a_{1,4}\rangle\rtimes
F\langle  a_{2,3},  a_{2,4} \rangle.     
$$ 
We write every element of ${PM}_{5}$  in the Markov normal form 
$$
f_3 f_2, \ \  \text{where} f_3\in F\langle a_{1,2},  a_{1,3},  a_{1,4}\rangle,
\  f_2\in F\langle  a_{2,3},  a_{2,4}\rangle
$$
and we define the universal invariant $\mu: {PM}_{5}\to\widehat{\mathcal{M}}_{5}$
by the formula
$$
\mu(f_3 f_2) = \mu_3(f_3) \, \mu_2(f_2),
$$
where $\mu_3$ and $\mu_2$ are defined as follows
$$
\mu_3(a_{1,j}) = 1+B_{1,j}, \ j=2, 3, 4,
$$
$$
\mu_2(a_{2,k}) = 1+ B_{2,k},  \ k=3, 4.
$$

\end{document}